\documentclass{amsart}
\usepackage{amssymb}
 
\newtheorem{theorem}{Theorem}[section]
\newtheorem{lemma}[theorem]{Lemma}

\newtheorem{corollary}[theorem]{Corollary}

\newtheorem{problem}[theorem]{Problem}

\theoremstyle{definition}
\newtheorem{definition}[theorem]{Definition}
\newtheorem{example}[theorem]{Example}

\newtheorem{algorithm}[theorem]{Algorithm}

\theoremstyle{remark}
\newtheorem{remark}[theorem]{Remark}

\numberwithin{equation}{section}

\newcommand{\lm}{\operatorname{lm}}

\newcommand{\lt}{\operatorname{lt}}

\newcommand{\<}{\langle}
\renewcommand{\>}{\rangle}

\DeclareSymbolFont{AMSb}{U}{msb}{m}{n}
\DeclareMathSymbol{\F}{\mathbin}{AMSb}{"46}
\DeclareMathSymbol{\N}{\mathbin}{AMSb}{"4E}
\DeclareMathSymbol{\Z}{\mathbin}{AMSb}{"5A}
\DeclareMathSymbol{\R}{\mathbin}{AMSb}{"52}
\DeclareMathSymbol{\C}{\mathbin}{AMSb}{"43}

\begin{document} \title[An Algorithms for Finding Symmetric Gr\"obner Bases]
{An Algorithm for Finding Symmetric \\ Gr\"obner Bases in Infinite Dimensional Rings}

\author{Matthias Aschenbrenner} \address{Department of Mathematics,
  University of California, Los Angeles, CA 90095.}%
\email{matthias@math.ucla.edu}%

\author{Christopher J. Hillar}
\address{Department of Mathematics, Texas A\&M University, College Station, TX 77843}
\email{chillar@math.tamu.edu}

\thanks{The first author is partially supported by the National
  Science Foundation Grant DMS 03-03618.  The work of the second
  author is supported under a National Science Foundation Graduate
  Research Fellowship.}

% \subjclass{13E05, 13E15, 20B30, 06A07}%
\keywords{Invariant ideal, partial ordering, symmetric group, Gr\"obner basis, polynomial reduction, algorithm}%

% \keywords{}

% ----------------------------------------------------------------
\begin{abstract}
A \textit{symmetric ideal} $I  \subseteq R = K[x_1,x_2,\ldots]$ is an ideal
that is invariant under the natural action of the infinite symmetric group.
We give an explicit algorithm to find Gr\"obner bases for symmetric
ideals in the infinite dimensional polynomial ring $R$.  This allows for symbolic computation
in a new class of rings.  In particular, we solve the ideal membership problem
for symmetric ideals of $R$.
\end{abstract} 

\maketitle 

\section{Introduction}
In computational algebra, one encounters the following general
problem.

\begin{problem}
Let $I$ be an ideal of a ring $R$ and let $f \in R$.  Determine whether
$f \in I$.
\end{problem}

When $R = K[x_1,\ldots,x_n]$ is a polynomial ring in $n$
indeterminates over a field $K$, this problem has a spectacular solution due to
Buchberger \cite{BB2} (for a nice exposition, see \cite{Cox2,Cox1}).

\begin{theorem}[Buchberger]\label{Buchberger}
Let $I = \<f_1,\ldots,f_m \>_R$ be an ideal of $R = K[x_1,\ldots,x_n]$.  Then, 
there is a computable, finite set of polynomials $G$ such that for every polynomial
$f$, we have $f \in I$ if and only if the polynomial reduction of $f$ with $G$ is $0$.
\end{theorem}

One remarkable feature of this result is that once such a \textit{Gr\"obner basis} $G$ for
$I$ is found, any new instance of the question ``Is $f \in I$''? can be solved
very quickly.  It is difficult not to stress the importance of Theorem \ref{Buchberger}; it
forms the backbone of the field of computational algebraic geometry and has
many applications, too numerous to list here.

We shall consider a different but related membership problem; one that at first glance 
would not  seem to be solvable as completely as Buchberger had done with 
$K[x_1,\ldots,x_n]$.
Let $X = \{x_1,x_2,\ldots\}$ be an infinite collection of
indeterminates, indexed by the positive integers, and let ${\mathfrak S}_{\infty}$ be the group of
permutations of $X$.  For a positive integer $N$, we will also let 
${\mathfrak S}_N$ denote the set of permutations of $\{1,\ldots,N\}$.
Fix a field $K$ and let $R = K[X]$ be the polynomial ring in the indeterminates $X$.  The group
${\mathfrak S}_{\infty}$ acts naturally on $R$: if $\sigma \in {\mathfrak
  S}_{\infty}$ and $f\in K[x_1,\dots,x_n]$, then 
\begin{equation}\label{groupaction}
\sigma f(x_1,\ldots,x_n) = f( x_{\sigma1},\dots,  x_{\sigma n})\in R.
\end{equation}

We motivate our discussion with the following concrete problem.
Questions of this nature arise in applications to chemistry  \cite{Ruch1, Ruch2, Ruch3}
and algebraic statistics \cite{SturmSull}.

\begin{problem}\label{motivprob}
Let $f_1 = x_1^3 x_3 + x_1^2 x_2^3$ and $f_2 = x_2^2 x_3^2 - x_2^2 x_1 + x_1 x_3^2$ and
consider the ideal of $R = K[X]$ generated by all permutations of $f_1$ and $f_2$:
\[ I = \<{\mathfrak  S}_{\infty}f_1, {\mathfrak  S}_{\infty}f_2\>_R.\]
Is the following polynomial involving $10$ indeterminates in $I$?
\begin{equation*} 
\begin{split}
f  \ & = -x_{10}^2x_{9}^2x_{5}^6-2x_{10}^2x_{9}x_{8}^3x_{5}^5-x_{10}^2x_{8}^6x_{5}^4
+3x_{10}^2x_{8}^2+3x_{10}^2x_{7}+3x_{10}x_{9}x_{7}x_{4}^3x_{3}^2x_{2}^2x_{1}\\ 
& +3x_{10}x_{9}x_{7}x_{4}^3x_{3}^2x_{1}^2-3x_{10}x_{9}x_{7}x_{4}^3x_{2}^2x_{1}^2
  -x_{9}^2x_{8}^7x_{7}x_{6}x_{5}^6-2x_{9}x_{8}^10x_{7}x_{6}x_{5}^5 \\ 
 & +x_{9}x_{5}^3x_{3}x_{2}x_{1}^3
 +x_{9}x_{5}^3x_{2}^4x_{1}^2+x_{9}x_{3}x_{2}^3x_{1}^4+x_{9}x_{2}^6x_{1}^3   -x_{8}^13x_{7}x_{6}x_{5}^4 
  -3x_{8}^2x_{7} \\ 
 & +x_{7}^2x_{6}x_{3}^3x_{2}^7+x_{7}^2x_{6}x_{3}^3x_{2}^5x_{1}-x_{7}^2x_{6}x_{3}x_{2}^7x_{1}+x_{5}x_{4}^2-3x_{5}x_{3}^2+2x_{5}x_{1}^2+x_{4}^2x_{3}^2 \\ & -2x_{3}^2x_{1}^2  +5x_{3}x_{1}^5+5x_{2}^3x_{1}^4.
\end{split}
\end{equation*}
More generally, given $f \in R$, how can we determine if $f \in I$?
\end{problem}

Naively, one could solve this problem using Buchberger's algorithm with truncated
polynomial rings $R_n = K[x_1,\ldots,x_n]$.
Namely, for each $n \geq 10$, compute
a Gr\"obner basis $G_n$ for the  ideal $I_n =  \<{\mathfrak  S}_{n}f_1, {\mathfrak  S}_{n}f_2\>_{R_n}$, 
and reduce $f$ by $G_n$.

There are several problems with this approach.  For one, this method requires
computation of many Gr\"obner bases (the bottleneck in any symbolic computation), the number
of which depends on the number of indeterminates appearing in $f$.  Additionally, it 
lacks the ability to solve new membership problems quickly, a powerful feature of Buchberger's
technique.  One might hope to at least restrict the number of Gr\"obner basis computations
in terms of the number of indeterminates appearing in $f$, however, the following simple example
should temper one's optimism a little.

\begin{example}\label{trunccounterex}
Let $I$ be the ideal generated by all permutations of $x_1 + x_2$ and $x_1 x_2$.
% = \<{\mathfrak  S}_{\infty}(x_1 + x_2), {\mathfrak  S}_{\infty}(x_1 x_2)\>_{R}$.  
Then, $I = \<x_1,x_2,\ldots\>_R$, but 
\[ x_1 \notin \<x_1 + x_2, x_1 x_2\>_{K[x_1,x_2]}.\]
\end{example}

Our main result in this paper is an effective algorithm that solves the general membership
problem for symmetric ideals (such as those appearing in Problem \ref{motivprob}) and
has all of the important features of Buchberger's method.  
It is the first algorithm of its kind that we are aware of (although it is similar in spirit to
Buchberger's original algorithm).  Before we state our theorem
explicitly (Theorem \ref{mainalgthm}), we develop some notation.

Let $R[{\mathfrak S}_{\infty}]$ denote the (left) group ring of ${\mathfrak S}_{\infty}$ over $R$ 
with multiplication given by $f\sigma\cdot g\tau = fg(\sigma\tau)$ for $f,g\in R$ and
$\sigma,\tau\in {\mathfrak S}_{\infty}$, and extended by linearity.
The action (\ref{groupaction}) naturally gives $R$ the structure of a (left) module over the
ring $R[{\mathfrak S}_{\infty}]$.  For instance, we have 
\[[x_1 (12) + x_2 (23)] \cdot (x_1 x_3 + x_2) = x_1 x_2 x_3 + x_1^2 + x_1 x_2^2 + x_2 x_3.\]
An ideal $I \subseteq R$ is called  \textit{symmetric} if \[ {\mathfrak S}_{\infty}I := \{\sigma f
: \sigma \in {\mathfrak S}_{\infty}, \ f \in I\} \subseteq I.\] 
Symmetric ideals are then simply the $R[{\mathfrak S}_{\infty}]$-submodules of
$R$.

Also, for the purposes of this work, we will use the following notation.
Let $B$ be a ring and let $G$ be a subset of a $B$-module $M$.  Then
$\<f: f \in G \>_{B}$ will denote the $B$-submodule of $M$ generated
by the elements of $G$.  This notation greatly simplifies expressing symmetric
ideals in terms of their generators.

\begin{example}
$I = \<x_1,x_2,\ldots\>_R$ is an invariant ideal of $R$.  Written as 
a module over the group ring $R[{\mathfrak S}_{\infty}]$, it has the
compact presentation $I= \<x_1\>_{R[{\mathfrak S}_{\infty}]}$.
\end{example}

We may now state our main theorem.

\begin{theorem}\label{mainalgthm}
Let $I = \<f_1, \ldots, f_m\>_{R[{\mathfrak S}_{\infty}]}$ be a symmetric ideal of $R$.
Then, there is a computable, finite set of polynomials $G$ such that for every polynomial
$f$, we have $f \in I$ if and only if the polynomial reduction of $f$ with $G$ is $0$.
\end{theorem}

We should remark here that the polynomial reduction appearing in Theorem
\ref{mainalgthm} is only a slight modification of the reduction in the context of
normal (finite dimensional) polynomial rings.  We will also call the sets $G$
appearing above \textit{Gr\"obner bases} for reasons which will be evident in the 
section that follows.

\begin{example}
The ideal $I = \<x_1^3 x_3 + x_1^2 x_2^3, 
x_2^2 x_3^2 - x_2^2 x_1 + x_1 x_3^2\>_{R[{\mathfrak S}_{\infty}]}$ from
Problem \ref{motivprob} has a Gr\"obner basis given by:
\[ G = {\mathfrak S}_{3} \cdot \{x_3 x_2 x_1^2, x_3^2 x_1 + x_2^4 x_1 - x_2^2 x_1, 
x_3 x_1^3, x_2 x_1^4, x_2^2 x_1^2\}.\]
Once $G$ is found, testing whether a polynomial $f$ is in $I$ is computationally fast;
for instance, one finds that $f \in I$ for the polynomial encountered in Problem \ref{motivprob}.
\qed
\end{example}

In Section \ref{gbinfinite}, we discuss the history of this problem and state 
some of the foundational results that are ingredients in the proof of Theorem \ref{mainalgthm}.
In particular, we discuss there an important partial order on monomials that respects
the action of the symmetric group.  Section \ref{symreduction} briefly reviews the notion of 
reduction that occurs in our more general context,
and finally, in Section \ref{algorithm}, we describe our algorithm.
To keep the paper as expository as possible, we have left out many of the (technical)
proofs that will appear in a much longer version of this paper.

\section{Gr\"obner Bases for Symmetric Ideals}\label{gbinfinite}

The following was proved recently in \cite{AH}.  It says that while ideals of $R = K[X]$
are too big in general, those with extra structure have finite presentations.

\begin{theorem}\label{onevarfinitegenthm}
Every symmetric ideal of $R$ is finitely generated as an $R[ {\mathfrak S}_{\infty}]$-module.  
In other words, $R$ is a Noetherian $R[{\mathfrak S}_{\infty}]$-module.
\end{theorem}

\begin{remark}
Symmetric ideals can be arbitrarily complex in the following sense.  For each $n$,
there are symmetric ideals of $R$ that cannot have fewer than $n$ $R[{\mathfrak S}_{\infty}]$-module
generators \cite{HW}.  Moreover, such ideals are not always monomial.
\end{remark}

Theorem \ref{onevarfinitegenthm} was motivated by finiteness questions in
chemistry \cite{Ruch1,Ruch2, Ruch3} and algebraic statistics \cite{SturmSull} 
involving chains of symmetric ideals $I_k$ ($k = 1,2,\ldots$)
contained in finite dimensional polynomial rings $R_k$. 
We refer the reader to \cite{AH} for more details.

In the course of proving Theorem \ref{onevarfinitegenthm}, it was shown that,
in a certain sense, a symmetric ideal $I$ has a finite
minimal Gr\"obner basis (see below for a review of these
concepts).  Moreover, the existence of such a set of generators solves
the ideal membership problem in $R$.  

\begin{theorem}
Let $G$ be a Gr\"obner basis for a symmetric ideal $I$.  Then $f \in I$ 
if and only if $f$ has normal form $0$ with respect to $G$.
\end{theorem}

The normal form reduction we are talking about here is a modification of the
standard notion in polynomial theory and Gr\"obner bases;  
we describe it in more detail below.
Unfortunately, the techniques used to
prove finiteness in \cite{AH} are nonconstructive and therefore do not
give methods for computing Gr\"obner bases in $R$.
Our main result is an algorithm for finding these bases.

\begin{theorem}
Let $I = \<f_1, \ldots, f_m\>_{R[{\mathfrak S}_{\infty}]}$ be a symmetric ideal of $R$.
There exists an effective algorithm to compute a finite minimal Gr\"obner basis for $I$.
\end{theorem}

\begin{corollary}
There exists an effective algorithm to solve the ideal membership problem
for symmetric ideals in the infinite dimensional ring $K[x_1,x_2,\ldots]$.
\end{corollary}

The following is a brief review of the Gr\"obner basis theory for symmetric
ideals necessary that we will need (see \cite{AH} for more details).
Let us first note that an infinite permutation acting on
a polynomial may be replaced with a finite one.
\begin{lemma}\label{infpermfiniteperm}
Let $\sigma \in {\mathfrak S}_\infty$ and $f \in R$.  Then there
exists a positive integer $N$ and $\tau \in {\mathfrak S}_N$ such that 
$\tau f = \sigma f$.
\end{lemma}

Let $\Omega$ be the set of monomials in indeterminates $x_1, x_2, \ldots$, 
including the constant monomial $1$.  Order the variables $x_1 < x_2 < \cdots$, and let 
$\leq$ be the induced lexicographic (total) well-ordering of monomials.  Given
a polynomial $f \in R$, we set lm$(f)$ to be the leading monomial of $f$ with
respect to $\leq$ and $\lt(f)$ to be its leading term.
The following partial ordering on $\Omega$ respects the action of ${\mathfrak S}_\infty$
and refines the division partial order on $\Omega$.

\begin{definition}\label{defpartialord}(The symmetric cancellation partial ordering)
\[v \preceq w \quad
:\Longleftrightarrow \quad \begin{cases} &\text{\parbox{200pt}{$v \leq
      w$ and there exist $\sigma \in {\mathfrak S}_\infty$ such that $\sigma v|w$ 
      and $\sigma u \leq \sigma v$ for all $u \leq v$.}}\end{cases}\]
\end{definition}

\begin{remark} A permutation $\sigma$ in the definition need
not be unique.  Also, we say that such a permutation \textit{witnesses} $v\preceq w$.
We will give a more computationally useful description of this partial order
in Theorem \ref{symordercharac} below.
\end{remark}

\begin{example}
As an example of this relation, consider the following
chain, \[x_1^3 \preceq x_1^2x_2^3 \preceq x_1^{\phantom{2}} x_2^{2} x_3^3.\]  To
verify the first inequality, notice that $x_1^2 x_2^3 = x_1^2 \sigma
(x_1^3)$, in which $\sigma$ is the transposition $(12)$.  If $u =
x_1^{u_1}\cdots x_n^{u_n} \leq x_1^3$, then it follows that $n =
1$ and $u_1 \leq 3$.  In particular, $\sigma u = x_2^{u_1}
\leq x_2^3  = \sigma x_1^3$.  Verification of the other inequality is similar.  

Alternatively, one may use Lemmas \ref{oneshiftuplem},  \ref{twoshiftuplem}, and
\ref{addtoendlem} to produce these and many other examples of such relations.  \qed
\end{example}

Although this partial order appears technical, it can be reconstructed from
the following two properties.  The first one says that the leading monomial 
of $\sigma f$ is the same as $\sigma \text{lm}(f)$ whenever there is a witness
$\sigma$ for lm$(f)$, while the latter can be viewed as a kind of  ``$S$-pair" leading term cancellation.

\begin{lemma}\label{cancellation}
Let $f$ be a nonzero polynomial and $w\in \Omega$.  Suppose that
$\sigma\in {\mathfrak S}_\infty$ witnesses $\text{\rm lm}(f)\preceq w$, and let $u\in \Omega$
with $u\sigma \text{\rm lm}(f)=w$. Then $\text{\rm lm}(u\sigma f)=u\sigma\text{\rm lm}(f)$.
\end{lemma}

\begin{lemma}
Suppose that $m_1 \preceq m_2$ and $f_1, f_2$ are two polynomials with 
lexicographic leading monomials $m_1$ and $m_2$, respectively.  
Then there exists a permutation $\sigma$
and $0 \neq c \in K$ such that \[ h = f_2 - c \frac{m_2}{\sigma m_1} \sigma f_1\] 
consists of monomials (lexicographically) smaller than $m_2$.
\end{lemma}

The following two lemmas allow us to generate many relations, including the
ones in the above example.  Proofs can also be found in \cite{AH}.

\begin{lemma}\label{oneshiftuplem}
Suppose that $x_1^{a_1}\cdots x_n^{a_n} \preceq x_1^{b_1}\cdots
x_n^{b_n}$ where $a_i,b_j\in\N$, $b_n>0$. Then for any $c \in \N$, we
have $x_1^{a_1}\cdots x_n^{a_n} \preceq x_1^c x_2^{b_1}\cdots
x_{n+1}^{b_n}$.
\end{lemma}

\begin{lemma}\label{twoshiftuplem}
  Suppose that $x_1^{a_1}\cdots x_n^{a_n} \preceq x_1^{b_1}\cdots x_n^{b_n}$,
  where $a_i,b_j\in\N$, $b_n>0$.   Then for any $a,b \in \N$ such that $a
  \leq b$, we have  $x_1^ax_2^{a_1}\cdots x_{n+1}^{a_n} \preceq
  x_1^bx_2^{b_1}\cdots x_{n+1}^{b_{n}}$.
\end{lemma}

The next fact is essentially a consequence of \cite[Lemma 2.14]{AH}.

\begin{lemma}\label{addtoendlem}
Let $u,v \in \Omega$ and set $n$ to be the largest index of indeterminates
appearing in $v$.  If $u \preceq v$, then there is a witness $\sigma \in {\mathfrak S}_n$,
and if $a,b \in \N$ are such that $a \leq b$, then $u x_{n+1}^{a} \preceq v x_{n+1}^{b}$.
\end{lemma}

In this setting, we need a notion of the leading monomials of a
set of polynomials that interacts with the symmetric group action.
For a set of polynomials $I$, we define
\[ \text{lm}(I) = \< w \in \Omega :  \text{lm}(f) \preceq w, \ 0 \neq f \in I\>_K,\]
the span of all monomials which are $\preceq$ larger than leading monomials
in $I$.  If $I$ happens to be a symmetric ideal, then it follows from
Lemma \ref{cancellation} that \[\text{lm}(I) = \<\text{lm}(f): f \in I \>_K\] corresponds
to a more familiar set of monomials.  With these preliminaries in
place, we state the following definition from \cite{AH}.

\begin{definition}
We say that a subset $B$ of a symmetric
ideal $I \subseteq R$ is a \emph{Gr\"obner basis}\/ for $I$
if lm$(B) = \text{lm}(I)$.
\end{definition}

Additionally, a Gr\"obner basis is called \textit{minimal} if
no leading monomial of an element in $B$ is $\preceq$ smaller
than any other leading monomial of an element in $B$.
%We should remark that the original notion of Gr\"obner basis was
%first introduced by Buchberger \cite{BB1,BB2} for the case of
%finitely many indeterminates $X$.
In analogy to the classical case, a Gr\"obner basis $B$
generates the ideal $I$:
\[ I =  \<B\>_{R[{\mathfrak S}_\infty]}.\]
The authors of \cite{AH} prove the following finiteness
result for symmetric ideals; it is an analog to the corresponding
statement for finite dimensional polynomial rings.  As a corollary,
they obtain Theorem \ref{onevarfinitegenthm}.

\begin{theorem}
A symmetric ideal of $R$ has a finite Gr\"obner basis.
\end{theorem}

Although much of the intuition involving Gr\"obner bases from the finite dimensional case 
transfers over faithfully to the ring $R$, one needs to be somewhat careful in general.  
For example, monomial generators do not automatically form a Gr\"obner basis for
a symmetric ideal $I$ (see Example \ref{surpriseGBex} below).
However, we do have a description of minimal Gr\"obner bases for
monomial ideals, and this is the content Theorem \ref{monGBprop} below.  
To state it, we need to introduce a special class of permutations to give a 
more workable description of the symmetric cancellation partial order.  This 
description will be used in our algorithm that finds symmetric Gr\"obner
bases.

Fix a monomial $g = \mathbf{x}^{\mathbf{a}} = x_1^{a_1}\cdots x_n^{a_n}$. 
A \textit{downward elementary shift} (resp. \textit{upward elementary shift})
%of $g$ is a transposition $\sigma$ which acts
of $g$ is a permutation $\sigma$ which acts
on $\mathbf{a}$ as transposition of two consecutive coordinates, the smaller (resp. larger)
of which is zero.  A \textit{downward shift} (resp. \textit{upward shift})
of $g$ is a product of downward elementary shifts (resp. upward elementary shifts) that begin
with $g$. A \textit{shift permutation} of $g$ is either a downward shift or an upward shift of $g$.
If $g,h \in \Omega$ and $\sigma$ is an upward shift of $g$ with $h = \sigma g$, then we
write $g \sim_{\sigma} h$.  
%For example, $\sigma = (34)$ is an upward elementary 
%shift of $g =  x_2^3x_3x_5^2$ and  $\tau = (32)(56)$(34) is an upward shift of $g$; 
%in this case, $g \sim_{\tau} h$ for $h = x_3^3x_4x_6^2$.
For example, $\sigma = (341)$ is an upward elementary 
shift of $g =  x_2^3x_3^{\phantom{3}}x_5^2$ and  $\tau = (32)(56)(341)$ is an upward shift of $g$; 
in this case, $g \sim_{\tau} h$ for $h = x_3^3x_4^{\phantom{3}}x_6^2$.

The following fact should be clear.

\begin{lemma}\label{translem}
If $g \sim_{\sigma} h$ and $h \sim_{\tau} k$, then $g \sim_{\tau \sigma} k$.
\end{lemma}

A more concrete description of these permutations is given by the
following straightforward lemma, which follows directly from the
definitions.
 
\begin{lemma}\label{shiftcharac}
Let $g$ be a monomial, and let $i_1 < \cdots < i_n$ be those 
indices appearing in the indeterminates dividing $g$.
Then $\sigma$ is an upward shift permutation of $g$ if and only if 
\[\sigma i_1 < \sigma i_2 <  \cdots < \sigma i_n \ \ \text{and} \ \  \sigma i_k \geq i_k, \ \ \ k = 1,\ldots,n.\]
\end{lemma}

The following fact gives a relationship between shift permutations and 
the symmetric cancellation partial order.

\begin{lemma}\label{invshiftlem}
Let $g$ and $h$ be monomials with $g \sim_{\sigma} h$ for some $\sigma \in {\mathfrak S}_\infty$. 
Then $g \preceq h$.  Moreover, we have $h \sim_{\sigma^{-1}} g$.
\end{lemma}
\begin{proof}
By transitivity and Lemma \ref{translem}, we may 
suppose that $\sigma$ as in the statement of the lemma acts on $g$ by
transposing $x_i$ and $x_{i+1}$.  
Write $g = x_1^{a_1} \cdots x_i^{a_i}x_{i+2}^{a_{i+2}} \cdots x_n^{a_n}$ with
$a_n > 0$; we must verify that 
\[ x_1^{a_1} \cdots x_i^{a_i}x_{i+2}^{a_{i+2}} \cdots x_n^{a_n}
\preceq x_1^{a_1} \cdots x_{i-1}^{a_{i-1}}x_{i+1}^{a_i} x_{i+2}^{a_{i+2}} \cdots x_n^{a_n}.\]
%so that $h = x_1^{a_1} \cdots x_{n-1}^{a_{n-1}}x_{n+1}^{a_n} x_{n+2}^{a_{n+2}} \cdots x_m^{a_m}$.
This is proved by induction on $n$.  When $n=1$, we have $i =1$, and the
claim reduces to Lemma \ref{oneshiftuplem}.  In general, we have two cases to consider.
If $i = n > 1$, then the claim follows from Lemma \ref{twoshiftuplem} and induction.
Alternatively, if $i < n$ and $n > 1$, then we may apply Lemma \ref{addtoendlem} and 
induction.  The second claim is clear from the definitions.
\end{proof}

\begin{remark}\label{carefulwitness}
A word of caution is in order.  Suppose that $g$ and $h$ are monomials with 
$g \sim_{\sigma} h$ for some  $\sigma \in {\mathfrak S}_\infty$.  Then it can happen that
$\sigma$ is \textit{not} a witness for the (valid) relation $g \preceq h$.  For example, if $\sigma = (14)(23)$, 
$g = x_2$, and $h = x_3$, then $g \sim_{\sigma} h$.  However, the relation $x_1 \leq x_2$
does not imply $\sigma x_1 \leq \sigma x_2$ as one can easily check.
\end{remark}

We now state a new characterization of the symmetric cancellation partial order.

%[****** The statement of the theorem is wrong...it needs to be adjusted*****]

\begin{theorem}\label{symordercharac}
Two monomials $v$ and $w$ satisfy  $v \preceq w$ if and only if
there is an upward shift $\sigma \in {\mathfrak S}_N$ of $v$ such that
$\sigma v | w$, where $N$ is the largest index of indeterminates appearing in $w$.
\end{theorem}

The main result of this section is the following.

\begin{theorem}\label{monGBprop}
Let $G$ be a set of $n$ monomials 
of degree $d$, and let $N$ be the largest index of indeterminates appearing
in any monomial in $G$.  Then $H = {\mathfrak S}_N G$ is a 
(finite) Gr\"obner basis for $I = \<G\>_{R[{\mathfrak S}_\infty]}$.
Moreover, if we let
\[S =\{ h \in H :  \text{ there exists $g  \in H \backslash \{h\}$ and  $\sigma \in 
{\mathfrak S}_N$ with $g \sim_{\sigma} h$}\},\] 
then $H \backslash S$ is a minimal Gr\"obner basis 
for $I$.
\end{theorem}
\begin{proof}
Let $G$, $H$, $S$, $N$, and $I$ be as in the statement of the theorem;  we first show that $H$ is a 
Gr\"obner basis for $I$. The inclusion $\text{lm}(H) \subseteq \text{lm}(I)$ is clear 
from the definition.  So suppose that $w \in \text{lm}(I)$ is a monomial; we must  show that 
$h \preceq w$ for some $h \in H$.  Set $w = u \sigma g$ for some monomial $u$, witness 
$\sigma \in {\mathfrak S}_\infty$, and $g \in G$.  Since  $\sigma g \preceq u\sigma g = w$,
it suffices to show that $h \preceq \sigma g$ for some $h \in H$.
Let $\tau$ be a downward shift that takes $\sigma g$ to a monomial $h$
with indices at most $N$.  Then $h$ has the same type (its unordered vector
of exponents) as $g$, and therefore there is a permutation $\gamma \in {\mathfrak S}_N$ such that
$h = \gamma g $. It follows that $h \in H$ and $h \sim_{\tau^{-1}} \sigma g$
so that $h \preceq \sigma g$ by Lemma \ref{invshiftlem}.

Next, we observe that $H \backslash S$ is still a Gr\"obner 
basis since $g \sim_{\sigma} h$ implies that $g \preceq h$.
Therefore, it remains to prove that $H \backslash S$ is minimal.
If $h,g \in H$ are related by $g \preceq h$, then $h = m \sigma g$ for 
a witness $\sigma$ and a monomial $m$.  Since each element of $H$
has the same degree, we have $m = 1$.  By Theorem \ref{symordercharac},
it follows that we may choose $\sigma \in {\mathfrak S}_N$ such that
$g \sim_{\sigma} h$.  Therefore, we are only removing unnecessary
elements from the Gr\"obner basis $H$ when we discard the monomials in
$S$. This completes the proof.
\end{proof}

\begin{corollary}
Let $G$ be a finite set of monomials, and let $N$ be the largest
index of indeterminates appearing in any monomial in $G$.  Then  ${\mathfrak S}_N G$
is a (not necessarily minimal) Gr\"obner basis for $I = \<G\>_{R[{\mathfrak S}_\infty]}$.
\end{corollary}

\begin{example}\label{surpriseGBex}
%The ideal $I = \<x_1^2 x_2^{\phantom{3}}\>_{R[{\mathfrak S}_\infty]}$ has a minimal Gr\"obner
%basis $G = \{x_1^2 x_2, x_1 x_2^2\}$.  Neither element of $G$ can be
%removed because although $x_1^{\phantom{3}} x_2^2 = (12)(x_1^2 x_2^{\phantom{3}})$, 
%we do not have $x_1^{\phantom{3}} x_2^2 \preceq x_1^2 x_2^{\phantom{3}}$ or 
%$x_1^{\phantom{3}} x_2^2 \preceq x_1^2 x_2^{\phantom{3}}$, as one
%can readily verify. \qed
The ideal $I = \<x_1^2 x_3^{\phantom{3}}\>_{R[{\mathfrak S}_\infty]}$ has a Gr\"obner
basis, \[H = \{x_1^{\phantom{3}} x_2^2, x_1^{\phantom{3}} x_3^2, x_1^2 x_2^{\phantom{3}}, 
x_2^{\phantom{3}} x_3^2, x_1^2 x_3^{\phantom{3}},x_2^2 x_3^{\phantom{3}}\}.\] 
However, it is not minimal.  Removing those elements that are the result 
of upward shifts, we are left with the following minimal Gr\"obner basis for $I$:  
$\{x_1^{\phantom{3}} x_2^2, x_1^2x_2^{\phantom{3}}\}$. 
\qed
\end{example}

\section{Reduction of polynomials}\label{symreduction}

Before describing our Gr\"obner basis algorithm, we must recall the ideas of reduction from
\cite{AH}.  Let $f\in R$, $f\neq 0$, and let $B$ be a set of nonzero polynomials
in $R$. We say that $f$ is \emph{reducible by $B$} if there exists
$g \in B$ such that  we have $\lm(g)\preceq \lm(f)$, witnessed by some $\sigma \in
{\mathfrak  S}_{\infty}$ and 
$$\lt(f) = a w \sigma \lt(g)$$
for some nonzero $a \in K$ and a monomial $w \in \Omega$ such that
$w \sigma \lm(g)=\lm(f)$.  In this case we write
$f\underset{B}\longrightarrow h$, where
$$h=f - \big(a w \sigma g\big),$$
and we say that $f$ \emph{reduces to $h$} by $B$.  We say that $f$ is
\emph{reduced} with respect to $B$ if $f$ is not reducible by $B$. By
convention, the zero polynomial is reduced with respect to $B$.
Trivially, every element of $B$ reduces to $0$.

The smallest quasi-ordering on $R$ extending the relation
$\underset{B}\longrightarrow$ is denoted by
$\underset{B}{\overset{*}\longrightarrow}$.  If $f,h\neq 0$ and
$f\underset{B}\longrightarrow h$, then $\lm(h)<\lm(f)$, by
Lemma~\ref{cancellation}.  In particular, every chain
$$h_0\underset{B}\longrightarrow h_1\underset{B}\longrightarrow h_2
\underset{B}\longrightarrow \cdots$$
with all $h_i\in R\setminus\{0\}$
is finite. (Since the term ordering $\leq$ is well-founded.) Hence
there exists $r\in R$ such that
$f\underset{B}{\overset{*}\longrightarrow} r$ and $r$ is reduced with
respect to $B$; we call such an $r$ a \emph{normal form} of $f$ with
respect to $B$.

\begin{lemma}\label{reduction}
  Suppose that $f\underset{B}{\overset{*}\longrightarrow} r$. Then
  there exist $g_1,\dots,g_n\in B$, $\sigma_1,\dots,\sigma_n\in {\mathfrak  S}_{\infty}$ and
  $h_1,\dots,h_n\in R$ such that
  $$f=r+\sum_{i=1}^n h_i\sigma_i g_i\quad \text{and}\quad \lm(f)\geq
  \max_{1\leq i\leq n}\lm(h_i\sigma_ig_i).$$
  \textup{(}In particular,
  $f-r\in \<B\>_{R[{\mathfrak  S}_{\infty}]}$.\textup{)}
\end{lemma}

%\begin{proof}
%  This is clear if $f=r$. Otherwise we have $f
%  \underset{B}\longrightarrow h
%  \underset{B}{\overset{*}\longrightarrow} r$ for some $h\in R$.
%  Inductively we may assume that there exist $g_1,\dots,g_n\in B$,
%  $\sigma_1,\dots,\sigma_n\in G$
%such that $\lm(g_i)\preceq\lm(f)$ for each $i$, witnessed by $\sigma_i$,
%  and $h_1,\dots,h_n\in R$ such that
%  $$h=r+\sum_{i=1}^n h_i\sigma_i g_i\quad \text{and}\quad \lm(h)\geq
%  \max_{1\leq i\leq n}\lm(h_i\sigma_ig_i).$$
%  There are also
%  $g_{n+1},\dots,g_{n+m}\in B$, $\sigma_{n+1},\dots,\sigma_{n+m}\in
%  G$, $a_{n+1},\dots,a_{n+m}\in A$ and $w_{n+1},\dots,w_{n+m}\in
%  X^\diamond$ such that $\lm(w_{n+i}\sigma_{n+i}g_{n+i})=\lm(f)$ for
%  all $i$ and
%\[ \lt(f) = \sum_{i=1}^m a_{n+i}w_{n+i}\sigma_{n+i}\lt(g_{n+i}), \qquad f
%= h + \sum_{i=1}^m a_{n+i} w_{n+i}\sigma_{n+i} g_{n+i}.\] Hence
%putting $h_{n+i}:=a_{n+i}w_{n+i}$ for $i=1,\dots,m$ we have
%$f=r+\sum_{j=1}^{n+m} h_j\sigma_j g_j$ and $\lm(f)>\lm(h)\geq
%\lm(h_j\sigma_jg_j)$ if $1\leq j\leq n$, $\lm(f)=\lm(h_j\sigma_jg_j)$
%if $n<j\leq n+m$.
%\end{proof}

\begin{lemma}\label{char GB}
  Let $I$ be a symmetric ideal of $R$ and $B$ be a set of nonzero
  elements of $I$. The following are equivalent:
\begin{enumerate}
\item $B$ is a Gr\"obner basis for $I$.
\item Every nonzero $f\in I$ is reducible by $B$.
\item Every $f\in I$ has normal form $0$. \textup{(}In particular,
  $I=\<B\>_{R[{\mathfrak  S}_{\infty}]}$.\textup{)}
\item Every $f\in I$ has unique normal form $0$.
\end{enumerate}
\end{lemma}
\begin{proof}
  The implications
  (1)~$\Rightarrow$~(2)~$\Rightarrow$~(3)~$\Rightarrow$~(4) are either
  obvious or follow from the remarks preceding the lemma.  Suppose
  that (4) holds. Every $f\in I\setminus\{0\}$ with
  $\lt(f)\notin\lt(B)$ is reduced with respect to $B$, hence has two
  distinct normal forms ($0$ and $f$), a contradiction. Thus
  $\lt(I)=\lt(B)$.
\end{proof}

\section{Description of the Algorithm}\label{algorithm}

We begin by describing a method that checks when two monomials are $\preceq$ comparable,
returning a permutation (if it exists) witnessing the relation.
This is accomplished using the characterization given by Theorem \ref{symordercharac}.
In this regard, it will be useful to view monomials in $R$ as vectors of 
integers $v = (v_1,v_2,\ldots)$ with finite support in $\mathbb N^{\omega}$.

\begin{algorithm}\label{vwcompalg}\mbox{}(Comparing monomials in the symmetric
cancellation order)\\
Input: Two monomials $v$ and $w$ with largest indeterminate in $w$ being $N$.\\
Output: A permutation $\sigma \in {\mathfrak S}_N$ if $v \preceq w$; otherwise, \textbf{\emph{false}}.
\begin{enumerate}
\item Set $t := 1$, $match := \{\}$;
\item For $i = 1$ to N:
\item[]\hspace{0.5cm} For $j = t$ to N:
\item[]\hspace{1.0cm} If $v_i \neq 0$ and $v_i \leq w_j$, then 
\item[]\hspace{1.5cm} $t := j+1$;
\item[]\hspace{1.5cm} $match := match \cup \{(i,j)\}$;
\item[]\hspace{1.5cm} Break inner loop;
\item[]\hspace{0.5cm} $t := \max\{i+1,t\}$;

\item If $match$ contains fewer elements than the support of $v$, return \textbf{\emph{false}};

\item For $j = N$ down to $1$:
\item[]\hspace{0.5cm} Set $i :=$ largest integer not appearing as a first coordinate in $match$;

\item[]\hspace{0.5cm} If $j$ is not a second coordinate in $match$, then $match := match \cup (i,j)$;
\item Return the permutation that $match$ represents;
\end{enumerate}
\end{algorithm}

\begin{remark}
One must be somewhat careful when constructing the witness $\sigma$.  Changing the recipe
given in the algorithm above might produce incorrect results.  See also Remark \ref{carefulwitness}.
\end{remark}

%In words, the algorithm is to find out whether $v$ is a substring of $w$ 
%(excluding $0$s in $v$) with the additional constraint that 

\begin{example}
Consider the vectors $v = (1,2,0,2)$ and $w = (0,3,4,1)$ representing monomials
$x_4^2 x_2^2 x_1$ and $x_4 x_3^4 x_2^3$ respectively.  
Then, Algorithm \ref{vwcompalg} will return false since $match = \{(1,2),(2,3)\}$ 
contains less than three elements after Step $(2)$.

On the other hand, running the algorithm on inputs $v = (3,2,0,0,5)$ and
$w = (5,1,4,6,9)$ will produce an output of $\{(1,1),(2,3),(3,2),(4,4),(5,5)\}$, which
correctly gives the witness $\sigma = (23)$ to the relation 
$x_1^3 x_2^2 x_5^5 \preceq x_1^5 x_2x_3^4 x_4^6 x_5^9$.
\end{example}

We also need to know how to compute a reduction of a polynomial $f$ by another 
polynomial $g$ (assuming that $f$ is reducible by $g$).  Given a witness $\sigma$, however,
this is calculated in Lemma \ref{cancellation}.  Specifically, we set
\begin{equation}\label{sgpoly}
SG_{\sigma}(f,g) = f - \frac{\lt(f)}{\sigma \lt(g)} \sigma g.
\end{equation}

Notice that when $\sigma = (1)$, the polynomial $SG_{\sigma}(f,g)$ resembles 
the normal $S$-pair from standard Gr\"obner basis theory.

The general case of reducing
a polynomial $f$ by a set $B$ is performed as follows; it is a modification of
ordinary polynomial division in the setting of finite dimensional polynomial rings.

\begin{algorithm}\label{reducefbyg}\mbox{}(Reducing a polynomial $f$ by an ordered  set of polynomials $B$)\\
Input: Polynomial $f$ and an ordered set $B = (b_1,\ldots,b_s) \in R^s$.\\
Output: The reduction of $f$ by $B$. 
\begin{enumerate}
\item Set $p := f$, $r := 0$, $divoccured := 0$;
\item While $p \neq 0$:
\item[]\hspace{0.5cm} i := 1;
\item[]\hspace{0.5cm} $divoccured := 0$;

\item[]\hspace{0.5cm} While $i \leq s$;
\item[]\hspace{1.0cm} $g := b_i$;
\item[]\hspace{1.0cm} If there exists a $\sigma$ witnessing $\lm(g) \preceq \lm(p)$, then
\item[]\hspace{1.5cm} $p := SG_{\sigma}(p,g)$;
\item[]\hspace{1.5cm} $divoccured := 1$;
\item[]\hspace{1.5cm} Break inner loop;
\item[]\hspace{1.0cm} Else, $i := i + 1$;
\item[]\hspace{0.5cm} If $divoccured = 0$, then 
\item[]\hspace{1.0cm} $r := r + \lt(p)$;
\item[]\hspace{1.0cm} $p := p - \lt(p)$;
\item Return $r$;
\end{enumerate}
\end{algorithm}

\begin{example}
Let $f = x_3^2 x_2^2+x_2 x_1$ and $B = (x_3 x_1+x_2 x_1)$.  Reducing $f$ by
$B$ is the same as reducing $f$ by $x_3 x_1+x_2 x_1$ twice as one can check.  The
resulting polynomial is $x_2^3 x_1 + x_2 x_1$.  
\end{example}

Before coming to our main result, we describe a truncated version of it.

\begin{algorithm}\label{symGBtruncalg}\mbox{}(Constructing a truncated Gr\"obner basis for a symmetric ideal)\\
Input: An integer $N$ and polynomials $F = \{f_1,\ldots,f_n\} \subset  K[x_1,\ldots,x_N]$.\\
Output: A truncated Gr\"obner basis for $I = \<f_1, \ldots, f_n\>_{R[{\mathfrak S}_{\infty}]}$. 
\begin{enumerate}
\item Set $F' := F$;
\item For each pair $(f_i,f_j)$:
\item[]\hspace{0.5cm} For each pair $(\sigma, \tau)$ of permutations in ${\mathfrak S}_{N}$:

\item[]\hspace{1.0cm} $h := SG_{(1)}(\sigma f_i, \tau f_j)$;
\item[]\hspace{1.0cm} Set $r$ to be the reduction of $h$ by ${\mathfrak S}_{N}B'$; 
\item[]\hspace{1.0cm} If $r \neq 0$, then $B' := B' \cup \{r\}$;

\item Return $B'$;
\end{enumerate}
\end{algorithm}

\begin{remark}
As we have seen, it is not enough to choose $N$ to be the largest indeterminate
appearing in $F$ (c.f. Remark \ref{trunccounterex}). 
\end{remark}

We call the input $N$ the \textit{order} of a truncated basis for $F$.

\begin{algorithm}\label{symGBalg}\mbox{}(Constructing a Gr\"obner basis for a symmetric ideal)\\
Input: Polynomials $F = \{f_1,\ldots,f_n\} \subset  K[x_1,\ldots,x_N]$.\\
Output: A Gr\"obner basis for $I = \<f_1, \ldots, f_n\>_{R[{\mathfrak S}_{\infty}]}$. 
\begin{enumerate}
\item Set $F' := F$, $i := N$;
\item While true:
\item[]\hspace{0.5cm} Set $F'$ to be a truncated Gr\"obner basis of $F$ of order $i$;
\item[]\hspace{0.5cm} If every element of $F'$ reduces to $0$ by ${\mathfrak S}_{N}F$, then return $F$;
\item[]\hspace{0.5cm} $F := F'$;
\item[]\hspace{0.5cm} $i := i + 1$;
\end{enumerate}
\end{algorithm}

\begin{example}
Consider $F = \{x_1+x_2, x_1 x_2\}$ from the introduction.  One iteration of 
Algorithm \ref{symGBalg} with $i = 2$ gives $F' = \{x_1+x_2, x_1^2\}$.
The next two iterations produce $\{x_1\}$ and thus the algorithm returns with this 
as its answer. 
\end{example}

% ----------------------------------------------------------------

\end{document}